\documentclass{article}

\usepackage[utf8]{inputenc} 		
\usepackage[english]{babel}		
\usepackage[pdftex,
paper = a4paper,
margin=2.0cm,
includefoot,
footskip=15pt]{geometry}
\usepackage{amsmath,amssymb,amsthm}	
\usepackage{graphicx}			
\usepackage{mathrsfs}
\usepackage{mathastext}
\usepackage{hyperref}
\usepackage{url}
\usepackage[title,toc,titletoc,page]{appendix}
\graphicspath{{./graphics/}}
\usepackage{epstopdf}
\usepackage{tikz-cd}
\usepackage{tikz}
\usetikzlibrary{matrix,arrows,calc,positioning,chains,intersections,decorations.pathmorphing} 
\usepackage[all,dvips]{xy} 
\usepackage{boxedminipage}
\usepackage{multicol}
\newtheorem{theo}{Theorem}[section]
\newtheorem{lem}[theo]{Lemma}
\newtheorem{prop}[theo]{Proposition}
\newtheorem{property}[theo]{Property}
\newtheorem{example}[theo]{Example}
\newtheorem{conjecture}{Conjecture}[section]
\newtheorem{defi}[theo]{Definition}

\newtheorem{remark}[theo]{Remark}

\newtheorem{notationremarque}[theo]{Remark-Notation}
\usepackage{color}			
\usepackage{courier}
\title{On Andreadakis equality for a subgroup of the McCool group}

\author{Abdoulrahim Ibrahim,\\[6pt]
IMAG, Univ Montpellier, CNRS.\\
Montpellier - 34090, France\\
e-mail: abdoulrahim.ibrahim@umontpellier.fr\\[6pt]
}

 \date{}

\begin{document}
	\maketitle						
	

	\begin{center}
		\textbf{Abstract}
	\end{center}
The McCool group $P\Sigma_n$ ($n\geq1$) has families of subgroups such as the ordinary pure braid group $P_n$, the upper triangular McCool group $P\Sigma^+_n$ and the partial inner automorphism group $I_n$. The generalized Andreadakis conjecture holds for $P\Sigma^+_{n}$ and $P_n$. In this paper, we prove a similar result for $I_n$ and also establish an isomorphism between $I_n$ and the inner automorphism group Int$(P\Sigma^+_{n+1})$ of $P\Sigma^+_{n+1}.$ 

\medskip



\section{ Introduction and Notations}
The IA-automorphism group ${IA}_n$ ($n\geq1$) is a subgroup of the automorphism group $Aut(F_n)$ of a free group $F_n=\langle x_1,\dots,x_n\rangle$ of rank $n$ generated by the following automorphisms \cite{magnus35}:
\begin{equation}
\xi_{k,s,t}(x_l)=\left\{ \begin{array}{cl} x_k \cdot\left[x_s, x_t\right] & \textrm{ if } k=l,\\  x_l & \textrm{ if } k\neq l \end{array}\right.
, \text{  $ $   $ $ }
\xi_{i,j}(x_l)=\left\{ \begin{array}{cl} x^{-1}_j \; x_i\;  x_j & \textrm{ if } l= i,\\ x_l & \textrm{ if } l\neq i \end{array}\right.
\label{Eq:1.1}
\end{equation}
where $1\leq i \neq j \leq n, \; 1\leq k,s,t\leq n$ and $k,s,t $ are distinct. In \cite{andreadakis}, Andreadakis defined a descending filtration called Andreadakis filtration on $IA_n$ ($n\geq1$):
\begin{eqnarray}
{IA}_n= \mathcal{A}_n(1) \supset  \mathcal{A}_n(2) \supset  \mathcal{A}_n(3) \supset \cdots \supset  \mathcal{A}_n(k)\supset  \mathcal{A}_n(k+1) \supset \cdots
\end{eqnarray}
where $\mathcal{A}_n(k)$ ($k\geq 1$) is the subgroup of $Aut(F_n)$ consisting of automorphisms acting trivially modulo the $k$-th term of the lower central series $\Gamma_n(k)$ of $F_n$. The subgroup $\mathcal{A}_n(k)$ ($k\geq 1$) also coincides in some cases with the $k$-th term of the lower central series $\Gamma_k(IA n)$ of $IA_n$. For example, by definition, $\mathcal{A}_n(1)=\Gamma_1(IA_n)=IA_n$ and $\mathcal{A}_n(2)=\Gamma_{2}(IA_n)$ for all $n\geq1$ by [\cite{bachmut}, Lemma 5]. Moreover, Andreadakis proved in \cite{andreadakis} that $\mathcal{A}_3(3)=\Gamma_3(IA_3)$ and $\mathcal{A}_2(k)=\Gamma_k(IA_2)$ for all $k\geq1$. Moreover, it was recently established by Satoh in \cite {satoh19} that $\mathcal{A}_n(3)=\Gamma_3(IA_n)$, improving the result of Pettet \cite {pettet} showing that $\Gamma_3(IA_n)$ has at most a finite index in $\mathcal{A}_n(3).$ In \cite{andreadakis}, Andreadakis conjecture that $\mathcal{A}_n(k)=\Gamma_k(IA_n)$ for all $n,k\geq1.$ But this conjecture is not true in general (cf. \cite{bar13, bar16}). Andreadakis' problem gives rise to variants for any $B$ subgroup of $IA_n$: for each integer $k \in \mathbb{N}^{\ast}$, one have inclusions $\Gamma_k(B)\subseteq B\cap \Gamma_k(IA_n)\subseteq B \cap \mathcal{A}_n(k)$ and it is natural to wonder if they are equalities (generalized Andreadakis conjecture for $B$). In this paper we prove the generalized Andreadakis conjecture for the partial inner automorphism group $I_n$, a subgroup of $IA_n$ generated by $\nu_{p,i}=\xi_{1,i}\cdots\xi_{p,i}$ ($1\leq i\leq p \leq n$).
\begin{theo}[]{(Theorem \ref{3.3.7}) }\\
Let $\Gamma_k(I_n)$ denote the $k$-th term of the lower central series of $I_n$. Then the subgroup $I_n$ of ${IA}_n$ verifies the Andreadakis equality:
	\begin{eqnarray}
	\forall k \geq 1, \text{  $ $  } \mathcal{I}_n(k):=I_n \cap \mathcal{A}_n(k)=\Gamma_{k}(I_n).\nonumber
	\end{eqnarray}
\end{theo}
On the other hand, we prove that $I_n$ is isomorphic to the inner automorphism group of the upper triangular McCool group $P\Sigma^+_{n}$, a well-known subgroup of $IA_n$ generated by all automorphisms $\xi_{i,j}$ with $1\leq j<i\leq n$.
\begin{theo}{(Theorem \ref{3.3.2})}\\
The application $P\Sigma^+_{n+1}\subset I_{n+1}\overset{q_n}{\twoheadrightarrow} I_n$ induces an isomorphism $\overline{P\Sigma}^+_{n+1}=P\Sigma^+_{n+1}/Z(P\Sigma^+_{n+1}) \cong I_n$, where $Z(P\Sigma^+_{n+1})$ is the center of $P\Sigma^+_{n+1}$.
\end{theo}
\section{Lower central series of a group}
In this section, we recall some basic definitions about the lower central series of a group $G$ with $e_G$ the identity element. We refer the reader to \cite{mikhailovpassi,robinson,thetheoryofnilpotent} for further details. Let us now define the commutator of two elements $g_1$ and $g_2$ of $G$ as $\left[g_1,g_2\right]:=g_1^{-1} g_2^{-1} g_1 g_2$. More generally, if $n$ is a non-zero integer and $g_1,g_2,\dots,g_n$ are $n$-elements of $G$; a simple commutator of weight $n\geq1$ is given by:
\begin{eqnarray}
\left[ g_1\right]&=&g_1    \nonumber\\
\left[g_1, g_2\right]&=&g^{-1}_1g^{-1}_2 g_1 g_2 \nonumber\\
\left[g_1, g_2,g_3\right]&=&\left[ \left[g_1, g_2\right], g_3\right]=\left[g_1, g_2\right]^{-1}g^{-1}_3 \left[g_1, g_2\right] g_3 \nonumber\\ 
\vdots &\vdots& \vdots \nonumber\\
\left[g_1, g_2, \dots, g_n\right]&=&\left[ \left[g_1, \dots,g_{n-1}\right], g_n\right]\nonumber
\end{eqnarray}
Below we give some identities based on the commutators, which are easy to verify (see \cite{hall50,hall76}):
\begin{property}
	Let $a, b,c$ belong to a group $G.$ Then 
	\begin{enumerate}
		\item $a^b:=a \left[a,b\right] = b^{-1}a b $ and  $\left[ a,b\right]=\left[ b, a\right]^{-1},$ 
		\item $\left[ab, c\right]=\left[a,c\right]^b \left[ b,c\right]$ and $ \left[a, b c \right]= \left[a, c \right] \left[a,b\right]^c,$
		\item $\left[a,b^{-1},c  \right]^b \left[b, c^{-1}, a  \right]^c \left[ c, a^{-1} , b\right]^a=e_G,$ (Witt-Hall identity ).
	\end{enumerate}
where $e_G$ is the identity element of $G$.
	\label{2.1.1}
\end{property}

\begin{defi} Let $G$ a group with $e_G$ its identity element.
\begin{enumerate}
\item One defines the $k$-th lower central series, denoted $\Gamma_k(G)$, of $G$ by the following relations:
\begin{eqnarray}
\Gamma_{1}(G):= G \;\text{ and } \text{   }  \Gamma_{k+1}(G)= \left[\Gamma_{k}(G),G \right], \text{   } k\geq 1.
\end{eqnarray}
where $\left[A,B \right]= \langle \left[a,b\right] \text{ $ :$   }  a \in A \text{ and } b \in B \rangle$ is the commutator group of two subgroups $A,B$ of $G$.
\item The group $G$ is called nilpotent if there exists an integer $r$ such that $\Gamma_r(G)= \{e_G\}.$ For a nilpotent group $G,$ the smallest $r$ such that $\Gamma_r(G)\neq \{e_G\}$ and $\Gamma_{r+1}(G)=\{e_G\}$ is called the nilpotent class of $G$ and the group $G$ is called nilpotent of class $r.$ The group $G/\Gamma_{k+1}(G)$ is a nilpotent group of class $k$ for $k \geq 1$ and it is called the $k$-th nilpotent group of $G$ and we denote it by $N_k(G):=G/\Gamma_{k+1}(G).$
\end{enumerate}
\end{defi}
\begin{notationremarque}{ $ $}\\
Note that $\Gamma_{k+1}(G)$ is a subgroup of $\Gamma_{k}(G)$ because $\Gamma_{k}(G)$ is a normal subgroup of $G$ for all $k\geq1,$ and that $G^{ab}=G/\Gamma_2(G)$ is the abelianization of $G$ because $\Gamma_{2}(G)$ is the derived subgroup of $G.$ From [\cite{robinson}, 5.1.11], it is easy to see that $\Gamma_k(G)/ \Gamma_{k+1}(G)$ ($k\geq1$) is abelian and we denote it by $gr^{k}(G)$ and $0$ as its neutral element. In the rest of this paper, we use the following notation. By $N\leq G,$ we mean that $N$ is a subgroup of $G.$ By $N\trianglelefteq G,$ we mean that $N$ is a normal subgroup of $G.$ Given $N\trianglelefteq G,$ we denote the class  of an element $g\in G$ in the quotient group $G/N$ by $\overline{g}.$ 
\end{notationremarque}

\begin{prop}{[\cite{hall76}, Theorem 10.2.3]}\\
Suppose that $G$ is generated by elements $g_1,\dots, g_r.$ Then $gr^{k}(G)$ $(k\geq 1)$ is generated by the simple commutators $\left[h_{1},\dots,h_{k}\right]\mod \Gamma_{k + 1}(G)$ where the $h_j$ are chosen among the elements $\overline{g}_1,\dots,\overline{g}_r \in G^{ab}$ and are not necessarily distinct. If $G$ is of finite type, then $gr^{k}(G)$ is an abelian group of finite type whose rank is denoted by $\phi_k(G)$ for each $k \geq 1.$
\label{2.1.11}
\end{prop} 
The next theorem about quotient groups of the lower central row of a free group of finite rank is fundamental.

\begin{theo}{[\cite{hall76}, Theorem 11.2.4]}\\
Let $F_n=\langle x_1,\dots,x_n \rangle$ be a free group of rank $n$ and let $\Gamma_{n}(k)$ be the $k$-th subgroup of its lower central series. Then $\mathcal{L}_n(k):=\Gamma_{n}(k)/\Gamma_{n}(k+1)$ is a free abelian group of finite generation with a basis consisting of basic commutators of weights $1,2,\dots,k$ called the Hall-basis and each element $f$ of $F_n$ can be written as	
\begin{eqnarray}
f= c^{\epsilon_1}_1\cdot c^{\epsilon_2}_2 \cdots c^{\epsilon_t}_t \mod \Gamma_{n}(k+1), \text{  $ $ } \epsilon_i= \pm 1, \forall i=1, \dots t .
\end{eqnarray}
where the $c_i$ are the basic commutators of weight at most $n$ arranged in order and the $\epsilon_i$ are integers. 
\label{2.1.15}
\end{theo}
\begin{example}{$ $}\\
\begin{enumerate}
	\item $\mathcal{L}_n(1)=F^{ab}_n$  with basis the classes $\overline{x_1},\dots \overline{x_n} \in F^{ab}_n$ of $x_1,\dots, x_n,$
	\item $\mathcal{L}_n(2)=\Lambda^2 F^{ab}_n$ with basis the classes of $\left[x_i, x_j\right]$ and $i>j.$
\end{enumerate}
\end{example}
A formula due to Witt \cite{witt} provides the rank of $\mathcal{L}_n(k)$ for all $k\geq 1$ and the formula is given by
\begin{eqnarray}
r_n(k):= \frac{1}{k}\sum \limits_{d|k} \mu(d) n^{\frac{k}{d} }
\label{rangdeL}
\end{eqnarray}
where $\mu$ is the Möbius function and $d$ runs through all positive divisors of $k.$ An excellent reference for all this is Hall’s books \cite{hall50,hall76}.\\

Let us  now consider $G$ as semidirect product $N\rtimes H$ of two subgroups $N$ and $H$ of this group. We say that $G=N\rtimes H$ is an almost-direct product if an action of $H$ on $N^{ab}$ is trivial. The following theorem describes the lower central series as well as the quotient groups of their lower central series of such a group.
\begin{theo} { [\cite{falkrandell}, Theorem 3.1] }\\
Let $G=N\rtimes H$ be an almost direct product. Then there exists a split exact sequence of abelian groups 
\begin{eqnarray}
\{0\} \rightarrow gr^k(N) \rightarrow gr^k(G) \rightarrow gr^k(H) \rightarrow \{0\} \text{  $ $   for each } k \geq 1.
\end{eqnarray}
In particular, we have for all $k\geq 1$ 
\begin{eqnarray}
gr^k(G)\cong gr^k(N) \oplus gr^k(H) \text{  $ $  and $ $ } \phi_k(G)=\phi_k(N)+\phi_k(H).
\end{eqnarray}
\label{2.1.16}
\end{theo}
\begin{remark}
A similar statement of Theorem \ref{2.1.16} is given in \cite{ihara,kohno}.
\end{remark}
Let us now see the result of applying Theorem \ref{2.1.16} for several free groups of finite rank. Given $G_1,G_2,\dots, G_k$ free groups of finite rank with $G_p=\langle x_{p,1},\dots, x_{p,{n_p}} \rangle$ ($1\leq p\leq k$), we consider an almost direct product of free groups, i.e. an iterated semidirect product
\begin{eqnarray}
G=G_k \rtimes_{\alpha_k} G_{k-1} \rtimes_{\alpha_{k-1}} \cdots \rtimes_{\alpha_3} G_2 \rtimes_{\alpha_2} G_1= \rtimes^k_{p=1} G_p 
\label{eqsemid7}
\end{eqnarray}
of free groups $G_p$ in which the action of $\rtimes^q_{p=1} G_p$ on $G^{ab}_r$ is trivial for $1 \leq q< r \leq k.$ Thus, the application of Theorem \ref{2.1.16} allows to write:
\begin{eqnarray}
gr^d(G)=\displaystyle\bigoplus^{k}_{p=1} gr^d(G_p) =\displaystyle\bigoplus^{k}_{p=1}\mathcal{L}_{n_p}(d), \text{  $ $   } \forall  d\geq 1.
\end{eqnarray}
The structure of $G=\rtimes^k_{p=1} G_p$ \eqref{eqsemid7} is given by a group homomorphism $\alpha_l: \rtimes^{l-1}_{p=1} G_p \rightarrow IA(G_l)$ where $IA(G_l)$ denotes the group of automorphisms of $G_l$ which act trivially on its abelianization. The group $G=\rtimes^k_{p=1} G_p$ has the presentation as follows \cite{Johnsonsomepopular} :
 \begin{eqnarray}
\langle x_{p,i} \; \; \; (1 \leq i \leq n_p, \; 1 \leq p \leq k) \: | \; x^{-1}_{p,i} x_{q,j} x_{p,i}=\alpha_q (x_{p,i})(x_{q,j}) \; \; 1\leq p < q \leq k, 1 \leq j \leq n_q \rangle \label{Eq:2.3}
\end{eqnarray} and its abelianization $G^{ab}$ is a free abelian group of rank $\displaystyle\sum^{k}_{p=1}n_p.$ We conclude this section to rewrite the above presentation \eqref{Eq:2.3} for an almost-direct product of free groups. Recall first that the IA-automorphism group of $F_n$, denoted ${IA}_n$, is generated by the automorphisms given by \cite{magnus35} :
\begin{equation}
\xi_{k,s,t}(x_l)=\left\{ \begin{array}{cl} x_k \left[x_s, x_t\right] & \textrm{ if } k=l,\\  x_l & \textrm{ if } k\neq l \end{array}\right.
, \text{  $ $   $ $ }
\xi_{i,j}(x_l)=\left\{ \begin{array}{cl} x^{-1}_j \; x_i\;  x_j & \textrm{ if } l= i,\\ x_l & \textrm{ if } l\neq i \end{array}\right.
\label{Eq:2.4}
\end{equation}
where $1\leq i \neq j \leq n, \; 1\leq k,s,t\leq n$ and $k,s,t$ are distinct. The following proposition describes a group presentation of $G=\rtimes^k_{p=1} G_p$.
\begin{prop}{[\cite{cohen}, Proposition 2.1 ]}\\
Let $G=\rtimes^k_{p=1} G_p$ be an almost-direct product of the free groups $G_p=\langle x_{p,1},\dots, x_{p,{n_p}} \rangle.$ Then $G$ admits a presentation with generators $x_{p,i}$ for $1\leq p \leq k , \; 1 \leq i \leq n_p$ and the following relations 
\begin{eqnarray}
x_{q,j}\cdot x_{p,i}= x_{p,i} \cdot x_{q,j} \cdot w^{p,q}_{i,j} , \; \; 1 \leq p<q\leq k, \; \;   1 \leq i \leq n_p, \;  1 \leq j \leq n_q
\label{Eq:2.5}
\end{eqnarray}
where $ w^{p,q}_{i,j} \in \Gamma_2(G)$ is a word in generators $x_{q,1},\cdots, x_{q,n_q}.$
\label{2.1.18}
\end{prop}
\begin{proof}  
By \eqref{Eq:2.3} the group $G$ has for relations
\begin{eqnarray}
x^{-1}_{p,i} \cdot x_{q,j}\cdot x_{p,i}&=&\alpha_q (x_{p,i})(x_{q,j}) \nonumber\\
x_{q,j} \cdot x_{p,i}&=&x_{p,i} \cdot \alpha_q (x_{p,i})(x_{q,j}) 
\label{Eq:2.6}
\end{eqnarray} 
where $\alpha_q (x_{p,i}) \in IA_{n_q}:=IA(G_q).$ It is clear that $w^{p,q}_{i,j}:=\alpha_q (x_{p,i})(x_{q,j})$ is a word in the generators $x_{q,1},\cdots, x_{q,n_q}.$ Since $\alpha_q (x_{p,i}) \in IA_{n_q},$ we have $\alpha_q (x_{p,i})=\xi^{\epsilon_1}_1\cdots \xi^{\epsilon_m}_m$ where $\epsilon_s\in\{\pm1\}$ and each $\xi^{\epsilon_r}_r$ is either $\xi_{k,s,t}$ or $\xi_{i,j}$ for $1\leq r \leq m. $
Note that \begin{eqnarray}
\xi_{j,i}(x_{q,j})= x_{q,j} \left[x_{q,i} , x_{q,j}\right] \text{  $ $   $ $ } \xi_{j,s,t}(x_{q,j})=x_{q,j} \left[x_{q,s} , x_{q,t}\right].
\end{eqnarray}
Thus, an induction on $m$ shows that $w^{p,q}_{i,j}:=x^{-1}_{q,j} \;\alpha_q(x_{p,i})(x_{q, j})$ is an element of $\Gamma_2(G).$ From \eqref{Eq:2.6} we thus arrive at $x_{q,j} x_{p,i}= x_{p,i} \; x_{q,j}w^{p,q}_{i,j}.$
\end{proof}

\section{ The McCool group}
The McCool group or pure welded braid group on $n$-strands $P\Sigma_n$ ($n\geq1$) (also known as the group of basis conjugating automorphisms of free group) is a subgroup of ${IA}_n$ generated by the following automorphisms \cite{humphries}:
\begin{equation}
\xi_{i,j}(x_l)=\left\{ \begin{array}{cl} x^{-1}_j \; x_i\;  x_j & \textrm{if } l= i,\\ x_l & \textrm{if } l\neq i .\end{array}\right.
\end{equation}
for all $1\leq i\neq j\leq n$. McCool proved that the following relations known as McCool relations determine a presentation of $P\Sigma_n$ \cite{mccool}:
\begin{eqnarray}
\left[\xi_{k,j}, \; \xi_{s,t}\right]  &=&\displaystyle{1\!\!1}  \text{ if } \; \{i,j\} \cap \{s,t\}= \emptyset \nonumber,\\
\left[\xi_{i,j}, \; \xi_{k,j}\right]&=&\displaystyle{1\!\!1} \text{ for } \; i,j, k \text{ distinct }\nonumber,\\
\left[\xi_{i,j}. \; \xi_{kj}, \; \xi_{i,k}\right]&=&\displaystyle{1\!\!1}  \text{ for } \; i,j, k  \text{ distinct }.
\label{Eq:2.7}	
\end{eqnarray}
The group $P\Sigma_{n}$ ($n\geq1$) contains a well-known subgroup. This is the upper triangular McCool group $P\Sigma^+_{n},$ the subgroup of $P\Sigma_{n}$ generated by the automorphisms $\xi_{i,j}$ with $1\leq j<i \leq n.$ The ranks of the quotient groups of the lower central series of $P\Sigma^+_n$ were computed in \cite{cpvw} but the ranks $\phi_k(P\Sigma_n)$ ($k\geq1$) of the quotient groups of the lower central series of the full McCool group $P\Sigma_n$ were not determined in general. However, we would like to refer the reader to Appendix \ref{sec:Groupe quotient de la série central descendante pour le groupe 1}, where the ranks $\phi_k(P\Sigma_n)$ have been determined in some cases. Moreover, it is known that $P\Sigma^+_n$ ($n\geq1$) can be realized as an iterated almost-direct product of free groups (see \cite{ccprassidis,cpvw}). For this reason, then, one can use Proposition \ref{2.1.18} to determine a group presentation of $P\Sigma^+_n$ ($n\geq1$). The reader is referred to Appendix \ref{sec:Factor groups of the lower central series for upper triangular McCool group} for more details. In \cite{cohen pruidze}, Cohen and Pruidze determined that the center of the upper triangular McCool group is infinite cyclic. For a group $G,$ let $Z(G)$ denote the center of $G$ and we denote the group $G/Z(G)$ by $\overline{G}.$ This group is isomorphic to the inner automorphism group $Inn(G)$ of $G.$ In a later section of this paper, we show in more detail that $\overline{P\Sigma}^+_n$ ($n\geq 2$) is isomorphic to the partial inner automorphism group $I_n$, another subgroup of $P\Sigma_n$ generated by $\nu_{p,i}=\xi_{1,i}\cdot\xi_{2,i}\cdots,\xi_{p,i}$ ($1\leq i \leq p \leq n$).

\section{ Andreadakis filtration and Johnson homomorphisms}
\subsection{Andreadakis filtration }
In this section, we define the filtration of Andreadakis of $IA_n.$ To define this notion, we first recall the definition of an N-series on a group $G$ introduced by Lazard in the 1950s. A good reference for this topic is Lazard's original paper \cite{lazard}. An N-series is by definition a decreasing filtration $$G=H_1 \supseteq H_2 \supseteq \dots \supseteq H_k \supseteq H_{k+1} \supseteq \dots$$ of $G$ by subgroups $H_1, H_2, \dots $ of $G$ verifying $ \left[H_p, H_q\right]\subseteq H_{p+q}$ for all $p,q \geq 1.$ For $q=1,$ the relations $\left[H_p, G \right] \subseteq H_{p+1}$ imply that $H_p$ is a normal subgroup in $G$ and also that $H_p/H_{p+1}$ is a subgroup of $Z(G/H_{p+1}),$ the center of $G/H_{p+1}.$ It is clear that the group $H_p/H_{p+1}$ is an abelian group since $\left[H_p, H_p\right]\subseteq H_{2p} \subseteq H_{p+1}.$ Let $Gr^p(H)=H_p/H_{p+1}$ for all $p \geq 1.$ We will denote it additively, namely for all $x, y \in H_p $ one has $\overline{xy}=\overline{x}+\overline{y}$ in $Gr^p(H).$ Let us form the direct sum
\begin{eqnarray}
Gr(G)=\bigoplus_{p\geq 1}Gr^p(H).
\end{eqnarray}
One can easily prove by the identities 2 and 3 Property \ref{2.1.1} that $Gr(G)$ has the structure of a graded Lie algebra, with a Lie bracket $\left[x, y\right]:=\overline{\left[x,y\right]}$ induced by the commutator of $G.$ The most famous example of an N-series on $G$ is the lower central series $\{\Gamma_k(G)\}_{k\geq 1}$ and its resulting Lie algebra $gr^{\ast}(G)=\displaystyle\bigoplus_{k\geq 1}gr^k(G)$ is the standard graded Lie algebra over $\mathbb{Z}.$ The lower central series is the smallest of an N-series, i.e., given any N-series $\{H_k\}_{k}$ on $G,$ we have the relations $\Gamma_k(G) \subseteq H_k$ for all $k\geq 1$, established by induction on $k.$ There is a canonical application $$gr^k(G) \longrightarrow Gr^k(H)  \text{  $ $ } (k\geq 1).$$ 
\begin{defi}
An N-series $\{H_k\}_k$ on a group $G$ is said separating if the intersection of the subgroups $H_k$ reduces to the identity element i.e. ${\displaystyle\bigcap_{k\geq 1} } H_k =\{e_G\}.$
The group $G$ is said to be residually nilpotent if the lower central series $\Gamma_k(G)$ of $G$ is separating, i.e. ${\displaystyle\bigcap_{k\geq 1} }\Gamma_k(G) =\{e_G\}.$
\end{defi}
We are now in a position to introduce the Andreadakis filtration. The action of $Aut(G)$ on the $k$-th nilpotent quotient $N_k(G)=G/\Gamma_{k+1}(G)$ ($k\geq1$) induces a group homomorphism $\lambda_k: Aut(G) \longrightarrow Aut(N_{k}(G))$ and its kernel $\mathcal{A}_G(k)$, consisting of automorphisms trivially acting on the k-th nilpotent quotient of $G$, is the set
\begin{eqnarray}
\mathcal{A}_G(k)&=&\{\phi \in Aut(G) \; | \; g^{-1}\; \phi(g) \in \Gamma_{k+1}(G)  ,\; \forall g \in G\}
\label{Eq:2.9}
\end{eqnarray}
The first term $IA(G):=\mathcal{A}_G(1)$ is the IA-automorphism group of $G$ also known as the Torelli group of $G.$ By construction, the groups $\mathcal{A}_G(k)$ are normal subgroups of $Aut(G).$ Andreadakis \cite{andreadakis} showed that:
\begin{itemize}
	\item [(A0)]  For all $k\geq1,$ $\mathcal{A}_G(k+1) \subset \mathcal{A}_G(k).$	
	\item [(A1)]For all $k,d\geq 1,$ $\left[\mathcal{A}_G(k), \mathcal{A}_G(d)\right]\subseteq \mathcal{A}_G(k+d).$ 
	\item [(A2)] For all $k,d\geq 1,$ $ \phi \in \mathcal{A}_G(k)$ and $g \in \Gamma_d(G), \; g^{-1} \; \phi(g)  \in \Gamma_{k+d}(G).$ 
	\item [(A3)]  If $\displaystyle\bigcap_{d\geq 1} \Gamma_d(G)=\{e_G\}$ $ $ then $ $ $\displaystyle\bigcap_{d\geq 1} \mathcal{A}_G(d)= \displaystyle{1\!\!1}_G.$  
\end{itemize}
From (A0)-(A1), there is thus an N-series
\begin{eqnarray}
IA(G):=\mathcal{A}_G(1) \supset \mathcal{A}_G(2) \supset \dots \supset \mathcal{A}_G(k) \supset \dots
\label{filration15}
\end{eqnarray}
on $IA(G).$ Consequently, one must have :
$$\mathcal{A}_G(k) \supseteq \Gamma_k(IA(G)) \text{  for all } k\geq 1.$$ 
The descending filtration $\{\mathcal{A}_G(k)\}_{k \geq 1}$ defined on $IA(G)$ is sometimes called Johnson filtration, but is due to Andreadakis \cite{andreadakis}.
\begin{defi}
The N-series $\{\mathcal{A}_G(k)\}_{k\geq 1}$ on $IA(G)$ is called here the Andreadakis filtration of $IA(G).$ We write $Gr^k(\mathcal{A}_G):= \mathcal{A}_G(k)/\mathcal{A}_G(k+1)$ for each $k\geq1.$  
\end{defi}
The groups $Gr^k(\mathcal{A}_G)$ ($k\geq 1$) admit an action of $Aut(G)/IA(G)$ defined as follows: since $\mathcal{A}_G(k)\trianglelefteq Aut(G),$ the group $Aut(G)$ naturally acts on $\mathcal{A}_G(k)$ by conjugation. Let $\overline{\phi} \in Gr^k(\mathcal{A}_G)$ be a representative of an automorphism $\phi \in \mathcal{A}_G(k)$ and let $\overline{\psi} \in Aut(G)/IA(G)$ be a representative of an automorphism $\psi \in Aut(G).$ The action of $\overline{\psi}$ on $\overline{\phi}$ is given by
\begin{eqnarray}
\overline{\psi} \cdot \overline{\phi}:= \overline{\psi\phi \psi^{-1}} .
\label{Eq:2.11}
\end{eqnarray}
The groups $gr^k(G)$ $(k\geq 1)$ also admit an action of $Aut(G)/IA(G)$ defined as follows: since $\Gamma_k(G)$ is a characteristic subgroup of $G$ (see [\cite{thetheoryofnilpotent}, Lemma 2.3] ), $Aut(G)$ naturally acts on $\Gamma_k(G)$ and then on $gr^k(G).$ Let $\phi \in Aut(G)$ and let $\overline{g}\in gr^k(G)$ be a representative of an element $g\in \Gamma_k(G).$ Let the action of $\phi$ on $\overline{g}$ be given by
\begin{eqnarray}
\phi \cdot \overline{g}:=\overline{\phi(g)}.
\label{Eq:2.10}
\end{eqnarray}
The restriction of this action \eqref{Eq:2.10} to $IA(G)$ is trivial by (A2). So one has an action of $Aut(G)/IA(G)$ on $gr^k(G).$ To illustrate all these concepts, we consider the case where $G$ is the free group $F_n$. We write $V,\Gamma_n(k), \mathcal{L}_n(k),$ $ IA_n,$ $\mathcal{A}_n(k)$ and $Gr^k(\mathcal{A}_n)$ for $F^{ab}_n,\Gamma_{k}(F_n),\text{   } gr^k(F_n),$ $IA(F_n),\text{   } \mathcal{A}_{F_n}(k)$  and $Gr^k(\mathcal{A}_{F_n})$ respectively, unless otherwise stated. Magnus, Witt and Hall have elucidated the structure of $gr^{\ast}(F_n)$ (see the books \cite{magnuskarrassolitar,hall50,hall76}) and we recall two of their results:
\begin{itemize}
\item[(1)] The group $F_n$ is residual nilpotent, i.e. $\displaystyle\bigcap_{k\geq 1}  \Gamma_n(k) =\{e_{F_n}\},$	\item[(2)] The Lie algebra $gr^{\ast}(F_n):=\displaystyle\bigoplus_{k\geq 1}\mathcal{L}_n(k)$ is isomorphic to a free Lie algebra of rank $n.$
\end{itemize}
Let us now turn to the Andreadakis filtration on ${IA}_n$. This filtration given by the subgroups $\mathcal{A}_n(k)$ has the property of $\Gamma_k(IA_n) \subseteq \mathcal{A}_n(k)$ for all $k \geq 1.$  These inclusions become equalities for some cases, viz.
\begin{description}
\item[(a)] $\Gamma_1(IA_n) = \mathcal{A}_n(1)$ by definition,
\item[(b)] $\Gamma_2(IA_n) = \mathcal{A}_n(2)$ for all $n\geq 2$ by [\cite{bachmut}, Lemma 5],
\item[(c)] $\Gamma_k(IA_2) = \mathcal{A}_2(k)$ for all $k \geq 2$ and $\Gamma_3(IA_3) = \mathcal{A}_3(3)$ by [\cite{andreadakis}, Theorems 6.1 and 6.2].
\end{description}
Satoh \cite{satoh19} recently proved that $\Gamma_3(IA_n)=\mathcal{A}_n(3)$ for all $n\geq 3.$ He thus improved the result of Pettet \cite{pettet} where she showed that $\Gamma_3(IA_n)$ has at most a finite index in $\mathcal{A}_n(3).$ It was conjectured by Andreadakis \cite{andreadakis} that $\Gamma_k(IA_n) = \mathcal{A}_n(k)$ for all $k \geq 3$ and $n \geq 3.$ But Bartholdi \cite{bar13,bar16} showed that this conjecture, known as the {\em Andreadakis conjecture} is not true in general. In contrast, Andreadakis' problem gives rise to variants for any $B$ subgroup of $IA_n$: for each integer $k \in \mathbb{N}^{\ast}$, one have inclusions $\Gamma_k(B)\subseteq B\cap \Gamma_k(IA_n)\subseteq B \cap \mathcal{A}_n(k)$ and it is natural to wonder if they are equalities (generalized Andreadakis conjecture for $B$). Significant progress on the generalized Andreadakis conjecture were recently made by Takao Satoh, Ştefan Papadima and Jacques Darné. It is known from \cite{satoh17} that the restricted Andreadakis filtration to $P\Sigma^+_n$ coincides with the lower central series of $P\Sigma^+_n,$ i.e., for all $n \geq 1 $ we have $$\mathcal{M}^+_n(k):= \mathcal{A}_n(k) \cap P\Sigma^+_n=\Gamma_k(P\Sigma^+_n) \text{  $ $ } (\forall k \geq 1).$$ In \cite{darne,papadima}, a similar result holds for the case of the pure braid group $P_n$, which is a subgroup formed from automorphisms of $ P\Sigma_n$ that leaves the word $x_1\cdots x_n \in F_n$ invariant (see for example \cite{barda03}). In other words we have for all $n \geq 1,$ $$\mathcal{P}_n(k):= \mathcal{A}_n(k) \cap P_n = \Gamma_k(P_n) \text{  $ $ } (\forall k \geq 1) .$$
The pure braid groups $P_n$ and the upper triangular McCool group $P\Sigma^+_n$ are both subgroups of $P\Sigma_n$ and both have an almost-direct product structure of free groups. It is natural to ask the question: what other subgroup $B$ of $P\Sigma_n$ that decomposes into an iterated almost-direct product of free groups satisfies this equality, which we call the \emph{"Andreadakis equality "},
\begin{eqnarray}
\mathcal{B}_n(k):= \mathcal{A}_n(k) \cap B= \Gamma_k(B) \text{ $ $ } (\forall k \geq 1)  \text{  $ $ } ?
\label{Eq:2.12}
\end{eqnarray}
The partial inner automorphism group $I_n$ which is the subgroup of $P\Sigma_n$ generated by $\nu_{p,i}= \xi_{1,i} \xi_{2,i} \cdots \xi_{p,i}$ ($1\leq i\leq p \leq n$) is an almost-direct product iterated free groups. The group $I_n$ is a good candidate to verify the Andreadakis equality \eqref{Eq:2.12}. Later, we will show that $I_n$ satisfies Andreadakis equality using a tool called Johnson homomorphism. The Johnson homomorphism introduced in the 1980s by Johnson \cite{johnson83,johnson85}, is a good tool to study the groups $Gr^k(\mathcal{B})=\mathcal{B}_n(k)/\mathcal{B}_n(k+1)$ associated to the Andreadakis filtration $\{\mathcal{B}_n(k)\}_{k\geq 1}$ on $B.$
\subsection{The Johnson homomorphism }
Recall here the $k$-th Johnson homomorphism of $Aut(F_n).$ We refer the reader to \cite{satoh16,johnson83,johnson85} for more details. Let $Hom_{\mathbb{Z}}(A,B)$ be the set of all homomorphisms from abelian groups $A$ to $B.$ Consider the homomorphism 
\begin{eqnarray}
\tau'_k :\mathcal{A}_n(k) \longrightarrow Hom_{\mathbb{Z}}(V, \mathcal{L}_n(k+1))  
\end{eqnarray}
defined by the formula $\sigma \longmapsto \tau_k(\sigma) :x \mod \Gamma_{n}(2) \mapsto x^{-1} \sigma(x) \mod \Gamma_n(k+2),  \text{ for } x \in V$
One can easily see that the kernel $ker(\tau'_k)$ of $\tau'_k$ is, by definition, $\mathcal{A}_n(k+1),$ (see \eqref{Eq:2.9}). Now let  $V^{\ast}:=$Hom$_{\mathbb{Z}}(V, \mathbb{Z})$ be the  dual group of $V$ with $\{x^{\ast}_1, \dots,x^{\ast}_n \}$ its dual basis. Thus Hom$_{\mathbb{Z}}(V, \mathcal{L}_n(k+1)) \cong V^{\ast} \displaystyle\bigotimes_{\mathbb{Z}} \mathcal{L}_n(k+1)$ as abelian groups (see \cite{bourbaki}.) Hence, one can obtain an injective homomorphism that we denote it by $\tau_k,$
\begin{eqnarray}
\begin{array}{ccccc}
\tau_k & : & Gr^k(\mathcal{A}_n) & \hookrightarrow &  V^{\ast} \displaystyle\bigotimes_{\mathbb{Z}} \mathcal{L}_n(k+1) \\
& & \overline{\sigma} & \mapsto & x^{\ast}_i \otimes \overline{x^{-1}_i \sigma(x_i)}. \\
\end{array}
\label{Johnson homomo}
\end{eqnarray}
In \cite{kawa}, it is shown that $\tau_k$ is not surjective for $k\geq 2$ but that
\begin{eqnarray}
\tau'_1:IA_n \longrightarrow V^{\ast} \bigotimes_{\mathbb{Z}} \Lambda^2 V 
\label{Eq:2.15}
\end{eqnarray}
induces an isomorphism $\tau_1:Gr^1(\mathcal{A}_n)=IA^{ab}_n \xrightarrow{\cong} V^{\ast} \bigotimes_{\mathbb{Z}} \Lambda^2 V$ as abelian groups. By restricting $\tau'_1$ \eqref{Eq:2.15} to $P\Sigma_n,$ we arrive at the following: 
\begin{lem}
	The abelianization $P\Sigma^{ab}_n$ of $P\Sigma_n$ is a free abelian group and its basis is given by the class $\overline{\xi_{ij}}$ for $1\leq i\neq j \leq n.$
	\label{2.26}
\end{lem}
\begin{proof}
The Magnus generators $\xi_{i,j}$ ($1\leq i\neq j \leq n$) \eqref{Eq:2.4} generates $P\Sigma_n,$ so that the $\overline{\xi_{ij}}$ are a generating family of $P\Sigma^{ab}_n.$ 
The homomorphism $\tau^{P\Sigma}_1:P\Sigma^{ab}_n \rightarrow IA^{ab}_n \xrightarrow{\cong} V^{\ast} \bigotimes_{\mathbb{Z}} \Lambda^2 V $ sends this generators $\overline{\xi_{ij}}$ onto a free family of $V^{\ast} \displaystyle\bigotimes_{\mathbb{Z}} \Lambda^2 V$. Hence $\tau^P_1$ is an isomorphism on its image.
\end{proof}$ $\\
The group $Aut(F_n)/{IA}_n$ (isomorphic to $GL(n,\mathbb{Z})$) acts on $Gr^k(\mathcal{A}_n)$ via the action given by \eqref{Eq:2.11} and also on $\mathcal{L}_n(k)$ via the action given by  \eqref{Eq:2.10}  for each $k\geq 1.$ The following result is well known.
\begin{lem}{[ \cite{satoh16} , Lemma 3.11]}
The monomorphism $$\tau_k:Gr^k(\mathcal{A}_n)\hookrightarrow   V^{\ast} \bigotimes_{\mathbb{Z}} \mathcal{L}_n(k+1)$$ is an $Aut(F_n)/{IA}_n$-equivariant homomorphism. 
\end{lem}
\begin{defi}
For each $k\geq1,$ the $Aut(F_n)/{IA}_n$-equivariant monomorphism $\tau_k$ is called the $k$-th Johnson homomorphism of $Aut(F_n).$
\end{defi}
Let $B$ denote a subgroup of $P\Sigma_n$ and let $\mathcal{B}(k)=\mathcal{A}_n(k) \bigcap B$ denote the Andreadakis filtration restricted to $B.$ The group $Gr^k(\mathcal{B})= \mathcal{B}(k)/\mathcal{B}(k+1)$ is a subgroup of $Gr^k(\mathcal{A}_n).$ Let $\tau^B_k$ denote the $k$-th Johnson homomorphism restricted to $Gr^k(\mathcal{B}).$ Next, define the following homomorphism $$\tau^{(B)}_1:=\tau^B_1 \circ \mathfrak{J}_1:gr^1(B) \rightarrow Gr^1(\mathcal{B})\hookrightarrow {IA}_n^{ab}$$ 
obtained by composing the canonical homomorphism $\mathfrak{J}_1:gr^1(B) \rightarrow Gr^1(\mathcal{B})$ induced by the inclusion $\Gamma_k(B)\subseteq \mathcal{B}(k)$  with the natural homomorphism $\tau^B_1: Gr^1(\mathcal{B})\hookrightarrow {IA}_n^{ab}.$ 

\begin{prop}
Let $B$ denote a subgroup of $P\Sigma_n$. If the homomorphism $\tau^{(B)}_1:gr^1(B) \rightarrow {IA}_n^{ab}$ is injective, then we have $\mathcal{B}(2)=\Gamma_2(B).$ \label{2.2.7}
\end{prop} 
\begin{proof}
Assume that $\tau^{(B)}_1:gr^1(B) \rightarrow {IA}_n^{ab}$ is injective. Thus we have $\Gamma_2(B)=\Gamma_2({IA}_n)\bigcap B.$ Since  $\Gamma_2({IA}_n)=\mathcal{A}_n(2)$ by [\cite{bachmut},{Lemma 5}], we deduce $\Gamma_2(B)=\Gamma_2({IA}_n)\bigcap B=\mathcal{A}_n(2)\bigcap B=\mathcal{B}(2).$
\end{proof}

\section{Andreadakis equality for partial inner automorphism group}

This section consists of three parts: in the first we show that $I_n$ has a very direct connection with the upper triangular McCool group; in the second we determine the structure of the standard graded Lie algebra $gr^{\ast}(I_n)$ over $\mathbb{Z}$; in the third we show an equality between the Andreadakis filtration restricted to I n and its lower central series. We first start by defining the partial inner automorphism group $I_n$.

\subsection{Partial inner automorphism group $I_n $ }
Bardakov and Neshchadim \cite{barda17} defined the subgroup $V_p$ ($2\leq p \leq n$) of $P\Sigma_{n}$ generated by $\nu_{p,i}=\xi_{1,i}\xi_{2,i} \cdots \xi_{p,i}$ which act on $F_n$ as follows:
$$\begin{array}{ccccc}
\nu_{p,i} & : & F_n & \longrightarrow & F_n\\
& & x_k & \longmapsto & \left\{ \begin{array}{cl} \ x^{-1}_i x_k \; x_i & \textrm{if } 1\leq k \leq p \\ x_k  & \text{if} \; p < k \leq n \end{array}\right.\\
\end{array}$$
where $1\leq i \leq p.$ Every automorphism $\nu_{pi}$ is an inner automorphism of $F_p=\langle x_1,\dots,x_p\rangle.$ The subgroup $V_p$ is isomorphic to the inner automorphism group $Inn(F_p)$ of $F_p.$ The free group $F_p$ has a trivial center for $p\geq2$ and hence $V_p=F_p.$ Thus $V_p$ is a free group of rank $p$ on $\{\nu_{p,1},\dots,\nu_{p,p} \}.$ The partial inner automorphism group $I_n$ is the union of the subgroups $V_2,\dots,V_n$ denoted $\langle V_2, V_3, \dots, V_n \rangle.$ The group $I_n$ $(n\geq 2)$ can be realized as an almost-direct product of free groups described as follows. It is the iterated semidirect product
\begin{eqnarray}
I_n= V_n \rtimes_{\eta_{n}} (V_{n-1} \rtimes_{\eta_{n-1}} (\dots (V_3 \rtimes_{\eta_{3}} V_2)\dots))
\label{semidirectproductI} 
\end{eqnarray}
The structure of this iterated semidirect product \eqref{semidirectproductI} is given by the homomorphism $\eta_{p}:\rtimes^{p-1}_{k=2} V_k \rightarrow IA(V_p)$  such that 
$$ \eta_{m}(\nu_{q,j})(\nu_{p,i})=\nu^{-1}_{q,j} . \nu_{p,i} . \nu_{q,j}= \left\{ \begin{array}{cl} \ \nu_{p,i} &  j=i  \\ \nu_{p,i} &  i>q\\ \nu^{-1}_{p,j} . \nu_{p,i} . \nu_{p,j}  &   j\neq i \text{ and }  i\leq q \end{array}\right.$$
where $1\leq i \leq p,$ $1\leq j \leq q$ and $2\leq q < p \leq n.$ Here $IA(V_p)$ is the group of automorphisms of $V_p$ acting trivially on its abelianization. The proposition \ref{2.1.18} immediately implies the following proposition.
\begin{prop}{[\cite{barda17}, Proposition 1]}\\
The group $I_n=\rtimes^{n}_{k=2} V_k$ has a finite presentation with generators $\nu_{p,i}$ for $ 2 \leq p \leq n,$ $ 1\leq i\leq p$ and with the following relations	
\begin{eqnarray}
\left[\nu_{p,i} , \; \nu_{q,j} \right]&=&\displaystyle{1\!\!1}  \textrm{ if } j=i \nonumber \\
\left[\nu_{p,i} , \; \nu_{q,j} \right]&=&\displaystyle{1\!\!1}  \textrm{ if }  i>q \nonumber\\
\left[\nu_{p,i} , \; \nu_{q,j} \right]&=& \left[\nu_{pi} , \nu_{pj} \right] \text{ if }  j\neq i \text{ and }  i\leq q 
\label{Eq:3.26}
\end{eqnarray}	
where $1\leq i \leq p,$ $1\leq j \leq q$ and $2\leq q < p \leq n.$ Moreover the abelianization $I^{ab}_n$ is a free abelian group of rank $\frac{n(n+1)}{2}-1,$ a basis being provided by the class $\overline{\nu_{p,i}}$ of $\nu_{p,i}.$
\label{8}
\end{prop}  
In \cite{cohen}, Cohen determined the structure of the cohomology ring of an almost direct arbitrary product of free groups. As a consequence [Theorem 3.1, \cite{cohen}], we know how to determine the structure of the cohomology ring $H^{\ast}(I_n,\mathbb{Z})$ of $I_n.$ The reader will find the description of $H^{\ast}(I_n,\mathbb{Z})$ in an Appendix \ref{chap: Anneau de cohomologie du groupe}. Moreover we remark that the subgroup $P\Sigma^+_{n}$ of $P\Sigma_n$ generated by all automorphisms $\xi_{i,j}$ with $i> j$ is contained in $I_n$ for all $n\geq1$. In fact, we can express each generator $\xi_{i,j}$ ($i> j$) of $P\Sigma^+_{n}$ via those of $I_n$ with the formulas 
\begin{eqnarray}
\xi_{i,j}=\nu^{-1}_{(i-1),j} \; \nu_{i,j} \text{ for all  } 1\leq j< i\leq n.
\end{eqnarray}
We also note that $I_n$ $(n\geq1)$ is the product $V_n\cdot P\Sigma^+_n$. Indeed, since $V_n$ is normal in $P\Sigma_n$ (see Appendix \ref{sec:Groupe quotient de la série central descendante pour le groupe 1}), the product $V_n\cdot P\Sigma^+_n$ is a subgroup of $P\Sigma_n,$ and we can obtain generators of this group from generators of $I_n,$ and vice versa. On the other hand, it is clear that we can consider the group $I_n$ as a subgroup of $I_{n+1}.$ Moreover $I_{n+1}$ is a semidirect product $V_{n+1}\rtimes I_n.$ We define the projection $q_n:I_{n+1}\twoheadrightarrow I_n$ from this semidirect product decomposition. It induces an isomorphism between $I_n$ and $Inn(P\Sigma^+_{n+1})$, the inner automorphism group of $P\Sigma^+_{n+1}$.
\begin{theo}
The application $P\Sigma^+_{n+1}\subset I_{n+1}\overset{q_n}{\twoheadrightarrow} I_n$ induces an isomorphism $Inn(P\Sigma^+_{n+1})\cong \overline{P\Sigma}^+_{n+1} \cong I_n$.
\label{3.3.2}
\end{theo}
\begin{proof}
The projection $q_n$ is the quotient of $I_{n+1}=V_{n+1}\cdot P\Sigma^+_{n+1}$ by $V_{n+1}.$ It thus induces
\begin{eqnarray}
P\Sigma^+_{n+1}/(V_{n+1}\cap P\Sigma^+_{n+1})\cong V_{n+1}\cdot P\Sigma^+_{n+1}/V_{n+1}\cong I_n.
\end{eqnarray}
On the other hand we have $V_{n+1}\cap P\Sigma^+_{n+1}=Z( P\Sigma^+_{n+1})$. Indeed, an element of this intersection must be an inner automorphism fixing $x_1,$ which implies that it must be in $\langle \nu_{n+1,1} \rangle.$ Conversely $\langle \nu_{n+1,1} \rangle$ is clearly in this intersection. Hence $V_{n+1} P\Sigma^+_{n+1}=\langle \nu_{n+1,1} \rangle.$ Since $Z( P\Sigma^+_{n+1})=\langle \nu_{n+1,1}=\xi_{2,1}\cdot \xi_{3,1} \cdots \xi_{n+1,1}\rangle$ by [\cite{cohen pruidze}, Proposition 2.3], we can deduce $V_{n+1} \cap P\Sigma^+_{n+1}=Z( P\Sigma^+_{n+1})$. Therefore, we have $Inn( P\Sigma^+_{n+1})=P\Sigma^+_{n+1}/Z( P\Sigma^+_{n+1}) \cong I_n$.
\end{proof}

\subsection{Factor groups of the lower central series for $I_n$ }
From \eqref{semidirectproductI}, one can define the short exact sequence of groups $$\{\displaystyle{1\!\!1}\}\rightarrow V_n \longrightarrow I_n \longrightarrow I_{n-1} \rightarrow \{\displaystyle{1\!\!1}\}.$$ The group $I_{n-1}$ acts by conjugation on $V^{ab}_n$ trivially. Namely the action by conjugation of $\nu_{j,q}$ on $\nu_{n,p}$ is given by
$$\nu^{-1}_{j,q}.\nu_{n,p}.\nu_{j,q}= \left\{ \begin{array}{cl} \ \nu_{n,p} & \textrm{ if } p=q  \textrm{ or } p>j \\ \nu_{n,p} . \left[\nu_{n,p} \;, \nu_{n,q} \right] & \text{ if }  p\neq q \text{ and }  p\leq j \end{array}\right.$$
where $q\leq j \leq n-1.$ By Theorem \ref{2.1.16}, there is a short exact sequence
$$\{0\} \rightarrow gr^k(V_n) \rightarrow gr^k(I_n) \rightarrow gr^k(I_{n-1}) \rightarrow \{0\} \text{ $ $ for all } k\geq 1.$$ 
By induction on $n,$ we have $gr^k(I_n)=\displaystyle\bigoplus^{n}_{m=2} gr^k(V_m)=\displaystyle\bigoplus^{n}_{m=2}\mathcal{L}_m(k)$ for all $k\geq 1.$ Hence we arrive at the following result.
\begin{prop}
 $gr^k(I_n)$ is a free abelian group whose the rank is
\begin{eqnarray}
\phi_k(I_n) = \sum \limits^{n}_{m=2} r_{m}(k),
\end{eqnarray}
where $r_{m}(k)$ is the rank of the free abelian group $\mathcal{L}_m(k)$ \eqref{rangdeL}.
\label{lie:3.9}
\end{prop}$ $\\
Now we shall determine the structure of the Lie algebra $gr^{\ast}(I_n)$ $(n\geq 2)$ over $\mathbb{Z}.$ Let $\mathfrak{y}_{pi}$ $(1\leq i \leq p \leq n)$ denote the image of $\nu_{p,i}$ in $gr^{1}(I_n).$ For $2\leq p \leq n,$ let $\mathfrak{Y}_p:=\{\mathfrak{y}_{p1}, \dots, \mathfrak{y}_{pp}\} $ and $L[\mathfrak{Y}_p]$ be the corresponding free Lie algebra of rank $p$ over $\mathbb{Z}.$ By [\cite{ihara}, Lemma 3.1.1], it is clear, first, that the almost-direct product $I_n=V_n \rtimes I_{n-1}$ induces a decomposition of the associated Lie algebras 
\begin{eqnarray}
gr^{\ast}(I_n)=gr^{\ast}(V_n) \bigoplus gr^{\ast}(I_{n-1}).
\end{eqnarray}
Since $V_n$ is the free group of rank $n$ on $\nu_{n1},\dots,\nu_{nn}$, $gr^{\ast}(V_n)$ is the free Lie algebra of rank $n$ over $\mathbb{Z}$ generated by the classes $\mathfrak{y}_{n1}, \dots \mathfrak{y}_{nn} \in gr^1(V_n)$ (by Prop \ref{2.1.11}) of $\nu_{n1},\dots,\nu_{nn}.$ Hence 
\begin{eqnarray}
gr^{\ast}(I_n)=L[\mathfrak{Y}_n] \bigoplus gr^{\ast}(I_{n-1}).
\label{decomposition26}
\end{eqnarray} 
The following lemma describes the structure of the Lie algebra $gr^{\ast}(I_n).$ 
\begin{lem}
The Lie algebra $gr^{\ast}(I_n)$ $(n\geq 2)$ is the quotient of the free Lie algebra $L[\mathfrak{Y}_p]$ over $\mathbb{Z}$ generated by elements $\mathfrak{y}_{pi}$ for $ 2 \leq p \leq n$ and $ 1\leq i\leq p$ modulo the following relations.
\begin{eqnarray}
\left[\mathfrak{y}_{pi}, \mathfrak{y}_{qj}\right]&=&0  \textrm{ if } j=i \nonumber \\
\left[\mathfrak{y}_{pi}, \mathfrak{y}_{qj}\right]&=&0  \textrm{ if }  i>q \nonumber\\
\left[\mathfrak{y}_{pi}, \mathfrak{y}_{qj} - \mathfrak{y}_{pj}\right]&=& 0  \text{ if }  j\neq i \text{ and }  i\leq q
\label{Eq:grI} 
\end{eqnarray}
where $1\leq i \leq p,$ $1\leq j \leq q$ and $2 \leq q < p \leq n.$
\label{Lie:3.9}
\end{lem}
\begin{proof}
First, by Prop.\ref{8}, the $\overline{\nu}_{pi}$'s generate $gr^{\ast}(I_n)$ and satisfy the relations \eqref{Eq:grI}. Denoting by $\mathcal{r}_n$ the quotient of the free Lie algebra over $\mathbb{Z}$ generated by the symbols $\mathfrak{y}_{pi}$'s modulo these relations. Let $q_n:\mathcal{r}_n \twoheadrightarrow gr^{\ast}(I_n)$ be the projection induced by $\mathfrak{y}_{pi} \mapsto \overline{\nu}_{pi}$ and let  $\pi_n:\mathcal{r}_n \rightarrow \mathcal{r}_{n-1}$ be the epimorphism defined by $\mathfrak{y}_{ni}\mapsto 0$ $(1\leq i \leq n).$ We shall show by induction on $n\geq 2$ that $q_n$ an isomorphism. Note that for $n = 2,$ it is obviously an isomorphism. For $n\geq 3,$ we have a commutative diagram
\begin{eqnarray}
\begin{tikzcd}
\{0\}\arrow{r} & \mathfrak{K}_n\arrow{r}\arrow{d}{k_n} & \mathcal{r}_n \arrow{r}{\pi_n}\arrow{d}{q_n \text{ (surjective) } } & \mathcal{r}_{n-1}\arrow{r}\arrow{d}{\text{isomorphism by induction assumption}} & \{0\} \\
\{0\}\arrow{r} & L[\mathfrak{Y}_n]\arrow{r}{ } & gr^{\ast}(I_n)\arrow{r} & gr^{\ast}(I_{n-1})\arrow{r} & \{0\}
\end{tikzcd}
\end{eqnarray}
where $\mathfrak{K}_n:=ker(\pi_n)$ and the second line comes from the decomposition \eqref{decomposition26}.The kernel $\mathfrak{K}_n$ of $\pi_n$ is an ideal of $\mathcal{r}_n$ generated by $\mathfrak{y}_{n1}, \dots \mathfrak{y}_{nn}.$ But in fact $\mathfrak{K}_n$ is generated by $\mathfrak{y}_{nj}$'s as a Lie sub-algebra. Since $L[\mathfrak{Y}_n]$ is free of rank $n$ on the $\mathfrak{y}_{nj}$ and $\mathfrak{K}_n$ is generated by these $\mathfrak{y}_{nj}$'s $(1\leq j \leq n).$ Hence $k_n$ must be an isomorphism. By induction assumption and using the five lemma, we deduce that $q_n$ is an isomorphism.
\end{proof}
\begin{remark}
A presentation of $gr^{\ast}(I_n)$ can also be found in [\cite{metaftsis}, Corollary 6], but the proof is based on different techniques than the one given here.
\end{remark}

\subsection{Andreadakis equality for $I_n$} 
Let $\mathcal{I}_n(k)$ be the k-th Andreadakis filtration restricted to $I_n$ defined by $\mathcal{I}_n(k):= \mathcal{A}_n(k) \bigcap I_n.$ For any $k\geq 1,$ let $Gr^k(\mathcal{I}_n):= \mathcal{I}_n(k) /\mathcal{I}_n(k+1)$ and let $\tau^I_k:=\tau_k|_{Gr^k(\mathcal{I}_n)} : Gr^k(\mathcal{I}_n) \hookrightarrow V^{\ast} \displaystyle\bigotimes_{\mathbb{Z}} \mathcal{L}_n(k+1)$ denote the monomorphism induced by the restriction of $\tau_k$ \eqref{Johnson homomo} to $Gr^k(\mathcal{I}_n).$ There exists a natural homomorphism $$\mathbb{g}_k: gr^k(I_n) \longrightarrow Gr^k(\mathcal{I}_n)$$ induced by the inclusion $\Gamma_{k}(I_n) \subseteq\mathcal{I}_n(k)$. We define a homomorphism $\tau^{(I)}_k$  to be the composition of $\mathbb{g}_k$ and $\tau^I_k$: 
\begin{eqnarray}
\tau^{(I)}_k:= \tau^I_k \circ \mathbb{g}_k:gr^k(I_n) \longrightarrow  V^{\ast} \displaystyle\bigotimes_{\mathbb{Z}} \mathcal{L}_n(k+1).
\end{eqnarray}
We will now determine the values of the image $Im(\tau^{(I)}_k )$ of $\tau^{(I)}_k .$ Recall that $I_n$ is generated by $\nu_{i,p}$ $(2\leq i \leq n, \; 1\leq p \leq i)$ acting on the
generators of $F_n$ according to the following rule 
\begin{eqnarray}
\nu_{i,p}(x_r)= \left\{ \begin{array}{cl} \ x^{-1}_p x_r \; x_p & \textrm{ if } 1\leq r \leq i \\ x_r  & \text{ if } \; r > i \end{array}\right.
\end{eqnarray}
Thus for $1\leq p_1, p_2 \leq i,$ we have
\begin{eqnarray}
(\nu_{i,p_1} \nu_{i,p_2})(x_r)=\nu_{i,p_1}(\nu_{i,p_2}(x_r))= \left\{ \begin{array}{cl} \ x^{-1}_{p_1}x^{-1}_{p_2} x_r \; x_{p_2} x_{p_1} & \textrm{ if } 1 \leq r \leq i \\ x_r  & \text{ if }  r > i \end{array}\right.
\end{eqnarray}
Consider  now a product given by $\Theta= \nu^{\epsilon_1}_{i,p_1}.\nu^{\epsilon_2}_{i,p_2}.\cdots.\nu^{\epsilon_k}_{i,p_k}$ where $\epsilon_t= \pm1$ and $1\leq p_t \leq i$ with $t \in \{1, \dots, k\} .$ Let us see how $\Theta$ acts on each generator $x_r \in F_n.$ The action of $\Theta$ on $x_r \in F_n$ is given by
\begin{eqnarray}
\Theta (x_r)= \left\{ \begin{array}{cl} \ (x^{\epsilon_1}_{p_1}. x^{\epsilon_2}_{p_2} .\cdots. x^{\epsilon_k}_{p_k})^{-1} \; x_r\; \;x^{\epsilon_1}_{p_1}. x^{\epsilon_2}_{p_2} .\cdots. x^{\epsilon_k}_{p_k} & \textrm{ if } 1\leq r \leq i \\ x_r  & \text{ if } \;  r>i  \end{array}\right.
\end{eqnarray}
Hence the action of commutator $\left[\nu_{i,p_1}, \; \nu_{i,p_2}  \right]\in \Gamma_{2}(I_n)$ on each generator $x_r$ is given by
\begin{eqnarray}
\left[\nu_{i,p_1}, \; \nu_{i,p_2}  \right](x_r)= \left\{ \begin{array}{cl} \left[x_{p_1}, \; x_{p_2}\right]^{-1} \cdot x_r \cdot \left[x_{p_1}, \; x_{p_2}\right]& \textrm{ if } 1\leq  r \leq i \\ x_r  & \text{ if } \;  r >i \end{array}\right. 
\end{eqnarray}
We next act the $k$-commutator $\mathfrak{T}=\left[\dots \left[\dots \left[\nu_{i,p_1},\; \nu_{i,p_2} \right] \dots \right],\nu_{i,p_k} \right] \in \Gamma_{k}(I_n)$ on each generator $x_r \in F_n$ and we obtain
\begin{eqnarray}
\mathfrak{T} (x_r)= \left\{ \begin{array}{cl} \left[\dots,\left[x_{p_1}, \; x_{p_2}\right] \dots ], x_{p_k}\right]^{-1} \cdot x_r \cdot \left[\dots, \left[x_{p_1}, \; x_{p_2}\right] \dots ], x_{p_k}\right]& \textrm{ if } 1\leq r \leq i \\ x_r  & \text{ if } r >i \end{array}\right. 
\end{eqnarray}
\begin{prop}
Let $\mathfrak{T}=\left[ \dots ,\left[\nu_{i,p_1}, \; \nu_{i,p_2}\right],\dots ] ,\nu_{i,p_k}\right]\in \Gamma_{k}(I_n).$ If $1 \leq p_1, p_2, \dots, p_k \leq i $ for $2\leq i \leq n,$ then
\begin{eqnarray}
\tau^{(I)}_k(\overline{\mathfrak{T}})=\left\{ \begin{array}{cl} x^{\ast}_r \otimes \left[x_{p_1}, \dots,x_{p_k}, x_r \right] & \textrm{ if } 1\leq r \leq i \\ x_r  & \text{ if } r >i \end{array}\right.
\end{eqnarray}
Moreover, $Im(\tau^{(I)}_k)$ is generated by
\begin{eqnarray}
\{ x^{\ast}_r \otimes \left[x_{p_1}, \cdots,x_{p_k}, x_r \right]    \in \; V^{\ast} \otimes_{\mathbb{Z}} \mathcal{L}_n(k+1) : \; 1 \leq p_1, p_2, \cdots, p_k ,r\leq i \text{  $ $ and $ $ } 2\leq i\leq n \}. 
\end{eqnarray} 
\end{prop}
For any $k\geq 1$ and $1\leq s\leq n,$ we denote by $\{ c^{s}_j(k)\}_{1\leq j\leq n^s_k}$ all basic commutators of weight $k$ among the components $\{x_1, \dots, x_{s} \}.$ As we know the basic commutators  of weight $k$ form a basis of the free abelian group $\mathcal{L}_s(k)$ (see Theorem \ref{2.1.15}). We arrive at the following result 
\begin{lem}
$\mathcal{H}:=\{ x^{\ast}_s \otimes \left[ c^{s}_1(k),x_s \right]  , \dots,  x^{\ast}_s \otimes \left[c^{s}_{n^{s}_k}(k), x_s \right] \; : \; 1\leq s \leq n  \}$ is a basis of $Im(\tau^{(I)}_k)$ as a free abelian group of rank $\sum \limits^{n}_{s=2} r_{s}(k)$ where $r_s(k)$ is the rank of the free abelian group  $\mathcal{L}_s(k).$ \label{3.3.6}
\end{lem}
\begin{proof}
First, it is clear that $\mathcal{H}$ generates $Im(\tau^{(I)}_k).$ It thus suffices to show that the elements of $\mathcal{H}$ are linearly independent. Assume that $$\sum \limits_{{s=1}}^n \sum \limits_{j=1}^{n^{s}_k} \lambda_{s,j} \;x^{\ast}_s\otimes \left[c^{s}_j(k), x_s\right]=0$$ for $\lambda_{s,j} \in \mathbb{Z}.$ For each $1 \leq s \leq n$ fixed, we then have $$\sum \limits_{j=1}^{n^{s}_k} \lambda_{s,j} \left[c^{s}_j(k), x_s \right]=0.$$
We consider the injective homomorphism $\beta:\mathcal{L}_s(k) \hookrightarrow \mathcal{L}_s(k+1).$ Since  the elements $\left[c^{s}_j(k), x_s \right]$ belong to $\mathcal{L}_s(k+1)$ and the elements $c^{s}_j(k)$ for $1\leq j \leq n^{s}_k$ are linearly independent in $\mathcal{L}_s(k),$ the elements $\left[c^{i}_j(k), x_r\right]$ are linearly independent in $\mathcal{L}_s(k+1).$ This show that $\lambda_{s,j}=0$ for $1\leq j \leq n^{s}_k.$ Thus the elements in $\mathcal{B}$ are linearly independent and hence $rank(\mathcal{H})= \sum \limits^{n}_{s=2} r_{s}(k).$
\end{proof}
The next statement is an affirmative answer to question \eqref{Eq:2.12} on Andreadakis equality for $I_n.$
\begin{theo}
The subgroup $I_n$ of ${IA}_n$ satisfies the Andreadakis equality:
\begin{eqnarray}
\forall k\geq 1, \text{  $ $   $ $ }\mathcal{I}_n(k)=\mathcal{A}_n(k)\bigcap I_n=\Gamma_{k}(I_n).
\end{eqnarray}
\label{3.3.7}
\end{theo}
\begin{proof}
We show this theorem by induction on $k\geq 1.$ By definition we have $\Gamma_{1}(I_n)=I_n= \mathcal{I}_n(1).$ The canonical homomorphism $\mathfrak{S}:I_n^{ab}\rightarrow P\Sigma^{ab}_n$ is injective because the classes $\overline{\nu}_{i,p}$ of $\nu_{i,p}$ (which generate $I_n$) are a generating family of $I_n^{ab}$ and the homomorphism $\mathfrak{S}$ sends  this generating family  $\overline{\nu}_{i,p}$ onto a free family of $P\Sigma_n^{ab},$ (see Lemme \ref{2.26}). Hence we have $\Gamma_2(I_n)=\Gamma_2(P\Sigma_n)\bigcap I_n.$ Since $\Gamma_2(P\Sigma_n)=\mathcal{A}_n(2)\bigcap P\Sigma_n$ by [\cite{satoh12},{Corollary 2.2.}], we deduce $\Gamma_{2}(I_n)= \mathcal{I}_n(2).$ Assume that we have $\Gamma_{k}(I_n)= \mathcal{I}_n(k)$ for each $k\geq 3.$ Then we have the surjective homomorphism $$\tau^{(I)}_k: gr^k(I_n) \xrightarrow[]{\mathbb{g}_k} Gr^k(\mathcal{I}_n)  \xrightarrow[]{\tau^{I}_k} Im(\tau^{I}_k)=Im(\tau^{(I)}_k).$$ 
By Lemma \ref{3.3.6}, $Im(\tau^{(I)}_k)$ has the same rank as $gr^k(I_n),$ as free abelian group. The homomorphism $\tau^{(I)}_k$ is then an isomorphism and hence $\mathbb{g}_k$ must be injective. Thus $\Gamma_{k+1}(I_n)= \mathcal{I}_n(k+1).$ 
\end{proof}

\appendix

\section{Lower central series of McCool groups}
\subsection{Factor groups of the lower central series for McCool group }
\label{sec:Groupe quotient de la série central descendante pour le groupe 1}
In this appendix, we examine the rank of $gr^k(P\Sigma_n)$ for all $k,n\geq 1.$ The abelian group $P\Sigma^{ab}_n$ is isomorphic to $\mathbb{Z}^{n(n-1)}$ with basis $\overline{\xi_{i,j}}$ $(1\leq i\neq j \leq n).$ Let $\{ \xi^{\ast}_{i,j} | \; 1\leq i\neq j \leq n  \} $ be the dual of the basis $\{ \overline{\xi_{i,j}} | \; 1\leq i\neq j \leq n  \}$ of $P\Sigma^{ab}_n.$ Since the first homology group $H_1(P\Sigma_n, \mathbb{Z})$ of $P\Sigma_n$ is equal to $P\Sigma^{ab}_n$, $H_1(P\Sigma_n, \mathbb{Z})$ is a free abelian group with basis $\overline{\xi_{i,j}}.$ Thus 
\begin{eqnarray}
H_1(P\Sigma_n, \mathbb{Z})=\bigoplus_{1 \leq i\neq j \leq n} \mathbb{Z} \; \overline{\xi_{i,j}}.
\end{eqnarray} 
Hence the first cohomology group $H^1(P\Sigma_n, \mathbb{Z}),$ of $P\Sigma_n$ is given by the dual basis
\begin{eqnarray}
\xi^{\ast}_{i,j} \left( \overline{\xi_{k,l}} \right)= \begin{cases}
1 & \text{  $ $ if } k=i \text{  $ $ and } j=l \\
0 & \text{  $ $ otherwise }
\end{cases}
\end{eqnarray}
Brownstein and Lee \cite{browlee} are the first who determined in the years 93, the cohomology group $H^k(P\Sigma_n, \mathbb{Z})$ of $P\Sigma_n$ for $k=1,2$ and conjectured a presentation of the cohomology ring $H^{\ast}(P\Sigma_n, \mathbb{Z})$ of $P\Sigma_n.$ This conjecture has been proved by Jensen, McCammond and Meier \cite{jensen mc}. In particular, they determined the rank of each $k$-th cohomology group $H^k(P\Sigma_n, \mathbb{Z})$ of $P\Sigma_n$ given by
\begin{eqnarray}
rank\left(H^k(P\Sigma_n, \mathbb{Z})\right)= \binom{n-1}{k} n^k. \label{Eq:3.3}
\end{eqnarray}
Let $V_n$ be the free subgroup of $P\Sigma_n$ of rank $n$ on $\nu_{n,1},\dots\nu_{n,n}.$ It is easy to see that $V_n$ (which is equal to $Inn(F_n)$) is normal in $P\Sigma_n.$ Namely, the generators $\xi_{i,j}$ of $P\Sigma_n$ act by conjugation on the generators of $V_n.$ The action of $\xi_{i,j}$ on $\nu_{n,p}$ for all $1\leq p \leq n$ is given by:
\begin{eqnarray}
\xi_{i,j} \cdot  \nu_{n,p}:=\xi^{-1}_{i,j} \;  \nu_{n,p} \; \xi_{i,j}= \left\{ \begin{array}{cl} \ \nu_{n,p} & \textrm{ if } p\neq i   \\ \nu^{-1}_{n,j}  \nu_{n,p} \;  \nu_{n,j}  & \text{ if }  p\neq i  \end{array}\right.
\label{Eq:3.4}
\end{eqnarray}
Since $\xi_{n,j}=\nu^{-1}_{(n-1),j} \nu_{n,j}$ for all $1\leq j \leq n$ and $\xi_{n-1,n}=(\xi_{1,n}.\dots. \xi_{n-2,n} )^{-1} \nu_{n,n},$ the group $OP\Sigma_{n}=P\Sigma_{n}/V_n$ is generated by the images of all the generators of $P\Sigma_{n-1}$ and of elements $\xi_{1,n},\dots ,\xi_{n-2,n}$ modulo $V_n.$ We have also $P\Sigma_n= V_n \rtimes OP\Sigma_n$  which is given by the short split exact sequence
\begin{eqnarray}
\{\displaystyle{1\!\!1}\}\longrightarrow V_n \longrightarrow P\Sigma_{n} \longrightarrow OP\Sigma_{n} \longrightarrow \{\displaystyle{1\!\!1}\}.
\end{eqnarray}
From \eqref{Eq:3.4}, the action by conjugation $OP\Sigma_{n}$ on $V^{ab}_n$ is trivial. By Theorem \ref{2.1.16}, we have
\begin{equation}
gr^k(P\Sigma_n)= gr^k(V_n) \bigoplus gr^k(OP\Sigma_{n}) \text{  $ $  for each } k \geq 1 
\end{equation}
So far, as far as I know, a basis of the abelian group $gr^k(P\Sigma_n)$ is not, as far as I know, yet known and thus the ranks $\phi_k(P\Sigma_n)$ of $gr^k(P\Sigma_n)$ are not yet generally determined for all $k\geq 2$ and $n\geq 4.$ We will now examine the ranks $\phi_k(P\Sigma_n)$ of $gr^k(P\Sigma_n)$ for some cases. We start in the case $n = 2$ and $n = 3$ (because $P\Sigma_{1}=\{\displaystyle{1\!\!1}\},$ there is nothing to say in this case). The groups $P\Sigma_{2} $ and $P\Sigma_{3}$ have simple structures and hence their associated ranks $\phi_k(P\Sigma_2)$ and $\phi_k(P\Sigma_3)$ respectively can be rapidly determined. \\
\textbf{The case $n=2.$} Since $OP\Sigma_{2}$ is the trivial group, $P\Sigma_{2}= V_2=\langle \nu_{21}, \nu_{12} \rangle$ is a free group of rank 2. Hence, it is immediate that $gr^k( P\Sigma_2)=\mathcal{L}_2(k),$ the free abelian group of rank $r_2(k).$ Thus 
\begin{equation}
\phi_k( P\Sigma_2)= r_2(k) \text{  $ $  for each } k \geq 1.
\label{Eq:3.7}
\end{equation}
\textbf{The case $n=3.$} The group $OP\Sigma_3$ has three generators $\{\overline{\xi_{1,2}}, \overline{\xi_{2,1}}, \overline{\xi_{1,3}} \}$ (modulo $V_3$). The generators of $P\Sigma_3$ satisfy, in total, the following nine relations:
\begin{equation}
\left[\xi_{1,2}, \; \xi_{3,2}\right]=\displaystyle{1\!\!1}  \nonumber,\\
\left[\xi_{1,3}, \; \xi_{2,3}\right]=\displaystyle{1\!\!1}\nonumber,\\
\left[\xi_{2,1}, \; \xi_{3,1}\right]=\displaystyle{1\!\!1} \nonumber,\\
\end{equation}
\begin{equation}
\left[\xi_{2,3}, \; \xi_{2,1}.\xi_{3,1}\right]=\displaystyle{1\!\!1}  \nonumber,\\
\left[\xi_{1,3}, \; \xi_{1,2}.\xi_{3,2}\right]=\displaystyle{1\!\!1}\nonumber,\\
\left[\xi_{1,2}, \; \xi_{1,3}.\xi_{2,3}\right]=\displaystyle{1\!\!1} \nonumber,\\
\end{equation}
\begin{equation}
\left[\xi_{3,2}, \; \xi_{2,1}.\xi_{3,1}\right] =\displaystyle{1\!\!1}  \nonumber,\\
\left[\xi_{3,1}, \; \xi_{1,2}.\xi_{3,2}\right]=\displaystyle{1\!\!1}\nonumber,\\
\left[\xi_{2,1}, \; \xi_{1,3}.\xi_{2,3}\right]=\displaystyle{1\!\!1} \nonumber,\\
\end{equation}    
All these relations degenerate in $OP\Sigma_{3}.$
It follows that $OP\Sigma_{3}= \langle \overline{\xi_{1,2}},\overline{\xi_{2,1}}, \overline{\xi_{1,3}} \rangle $ is a free group of rank 3 and we arrive at the following result
\begin{lem}
$P\Sigma_{3}$ is an almost-direct product of two copies of free group $V_3 =F_3$ and $OP\Sigma_{3}=F_3$ of rank $3.$ That is, $$P\Sigma_{3}=V_3 \rtimes OP\Sigma_{3}.$$
\end{lem}$ $\\
As a consequence Theorem \ref{2.1.16}, we have the following result:
\begin{prop}
	There is an exact short sequence split		
\begin{eqnarray}
\{0\}\rightarrow gr^{k}(V_{3}) \rightarrow gr^{k}(P\Sigma_{3}) \rightarrow gr^{k}(OP\Sigma_{3}) \rightarrow \{0\}.
\end{eqnarray}
In other words, we have
\begin{eqnarray}
gr^{k}(P\Sigma_{3})=gr^{k}(V_{3}) \bigoplus gr^{k}(OP\Sigma_{3}) = \mathcal{L}_3(k) \bigoplus \mathcal{L}_3(k)  \text{  $ $  for each } k \geq 1
\end{eqnarray}
where $\mathcal{L}_3(k)$ is the free abelian group of rank $r_3(k).$
\label{3.1.1}
\end{prop}
From Proposition \ref{3.1.1}, we then deduce that
\begin{equation}
\phi_k(P\Sigma_{3})=2\cdot r_3(k) \text{  $ $  for each } k\geq 1
\end{equation}
\textbf{The case $n\geq 4.$} The structure of $OP\Sigma_n$ is quite complicated, it is difficult to determine the rank $\phi_k(OP\Sigma_n)$ of $gr^k(OP\Sigma_n)$ and especially that of $gr^k(P\Sigma_n).$ However for $k=1$ it is not surprising that $gr^1(P\Sigma_{n})=P\Sigma^{ab}_{n}=\mathbb{Z}^{n(n-1)}$ and thus
\begin{equation}
\phi_1(P\Sigma_{n})=(n-1)\cdot r_n(1).
\end{equation}
We will calculate the rank $\phi_2(P\Sigma_{n})$ of $gr^2(P\Sigma_{n})$ and at the same time we determine a basis for $gr^2(P\Sigma_{n}).$ To determine the rank $\phi_2(P\Sigma_{n}),$ we will apply the exact 5-term homology sequence. Recall that if we have an extension of group $\{e_G\} \rightarrow N \rightarrow G \rightarrow Q \rightarrow \{e_G\}$  then there exists an exact sequence of the form (see [\cite{stallings},Theorem 2.1 ])
\begin{eqnarray}
H_2(G) \rightarrow H_2(Q) \rightarrow N/\left[G, N\right] \rightarrow H_1(G) \rightarrow H_1(Q) \rightarrow \{e_G\}.
\end{eqnarray}
We apply this to the exact short sequence $\{\displaystyle{1\!\!1}\} \rightarrow \Gamma_{2}(P\Sigma_{n}) \rightarrow P\Sigma_{n} \rightarrow P\Sigma^{ab}_{n} \rightarrow \{\displaystyle{1\!\!1}\}.$ Thus we obtain
\begin{equation}
H_2(P\Sigma_{n},\mathbb{Z}) \rightarrow H_2(P\Sigma^{ab}_{n}, \mathbb{Z}) \rightarrow gr^2(P\Sigma_{n})\rightarrow H_1(P\Sigma_{n}, \mathbb{Z}) \rightarrow P\Sigma^{ab}_{n} \rightarrow \{\displaystyle{1\!\!1}\}
\label{Eq:3.10}
\end{equation}
Or $H_1(P\Sigma_{n}, \mathbb{Z})=P\Sigma^{ab}_{n} ,$ then  \eqref{Eq:3.10}  becomes
\begin{equation}
H_2(P\Sigma_{n}, \mathbb{Z}) \rightarrow H_2(P\Sigma^{ab}_{n}, \mathbb{Z}) \rightarrow gr^2(P\Sigma_{n}) \rightarrow \{\displaystyle{1\!\!1}\} \label{Eq:3.11}
\end{equation} 
Therefore we come to this following result.
\begin{theo}
For all $n\geq 1,$ $gr^2(P\Sigma_{n})$ has for possible basis
\begin{eqnarray}
\{[ \xi_{ij}, \xi_{ji} ] \; | 1 \leq i < j \leq n \} \; \cup \{[ \xi_{ij}, \xi_{it} ]  | 1\leq i\leq n, i\neq j,t, 1\leq j<t \leq n\}.
\end{eqnarray} whose its rank is 
\begin{equation}
	\phi_2(P\Sigma_{n})=(n-1)\cdot r_n(2).
	\end{equation} 
	\label{3.1.2}
\end{theo}
\begin{proof}
First of all, it follows from \eqref{Eq:3.11} and \eqref{Eq:3.3} that the rank $\phi_2(P\Sigma_n)$ of $gr^2(P\Sigma_{n})$  is at least to
\begin{eqnarray}
\binom{n(n-1)}{2}-\binom{n-1}{2} n^2 =\frac{n(n-1)^2}{2}=(n-1)\cdot r_n(2).
\label{Eq:A.19}
\end{eqnarray}
As indicated in [\cite{barda03}, Lemma 7], each of these two following subgroups of $P\Sigma_n$
\begin{equation}
	\langle \xi_{i,j}, \xi_{j,i}\rangle_{1\leq i < j \leq n} \text{  $ $ and } \langle \xi_{i,j}, \xi_{i,t}\rangle_{1\leq i\leq n, i\neq j,t, 1\leq j<t \leq n } \label{Eq:3.13}
	\end{equation}
is free groups of rank 2. By the Prop.\ref{2.1.11}, $gr^2(P\Sigma_n)$ is generated by $\left[\xi_{i,j}, \xi_{s,t} \right] \mod \Gamma_{3}(P\Sigma_{n})$ for indices $i,j,s,t$ distinct 2 to 2. It should be noted that from \eqref{Eq:3.13} 
	\begin{equation}
	0 \neq \left[\xi_{i,j}, \xi_{j,i}  \right] \in gr^2(P\Sigma_{n}) \text{  $ $ and } 0\neq  \left[\xi_{i,j}, \xi_{i,t}  \right] \in gr^2(P\Sigma_{n}). \label{Eq:3.14}
	\end{equation}
	We claim that $gr^2(P\Sigma_{n})$ is generated by the elements of \eqref{Eq:3.14}. Indeed, by the McCool relations \eqref{Eq:2.7}, we have
	\begin{itemize}
		\item If $\{i,j\} \cap \{s,t\} = \emptyset$, then $\left[\xi_{i,j}, \xi_{s,t}\right]=\displaystyle{1\!\!1}$ and thus $\left[\xi_{i,j}, \xi_{s,t}\right]=0$ in  $ gr^2(P\Sigma_n).$
		\item if $t=j$ then $\left[\xi_{i,j}, \xi_{s,j}\right]=\displaystyle{1\!\!1}$ and thus $\left[\xi_{i,j}, \xi_{s,j} \right]=0$ in $ gr^2(P\Sigma_n).$
		\item if now $i=t$, we then have $\left[\xi_{i,j}, \; \xi_{s,i} \right]= \left\{
		\begin{array}{ll}
		\left[\xi_{i,j}, \; \xi_{j,i} \right] & \mbox{ if } s=j  \\
		\displaystyle{1\!\!1} & \mbox{ otherwise,}
		\end{array}
		\right.$\\  and so we get
		$$\left[\xi_{i,j}, \; \xi_{s,i} \right]= \left\{
		\begin{array}{ll}
		\left[\xi_{i,j}, \; \xi_{j,i} \right] & \mbox{ if } s=j  \\
		0 & \mbox{ otherwise,}
		\end{array} \right. \text{ in } gr^2(P\Sigma_n).$$
		\item If now $s=j$, we then have $\left[\xi_{i,j}, \; \xi_{j,t}\right],$  or by 4 of the Property \ref{2.1.1} we have $$\displaystyle{1\!\!1}=\left[\xi_{i,j}, \; \xi_{i,t}.\xi_{j,t}\right] =\left[ \xi_{i,j},\; \xi_{j,t} \right] \left[ \xi_{i,j},\; \xi_{i,t} \right] \left[\left[ \xi_{i,j},\; \xi_{i,t} \right],\; \xi_{j,t}\right]$$ and thus  $\left[ \xi_{i,j},\; \xi_{j,t} \right]=-\left[ \xi_{i,j},\; \xi_{i,t} \right]$  in  $gr^2(P\Sigma_n).$
	\end{itemize}  
	Hence all generators are reduced to
	$[ \xi_{ij}, \xi_{it} ] \text{ $ $ and }\; [ \xi_{ij}, \xi_{ji} ].$ Hence, the rank $\phi_2(P\Sigma_n)$ of $gr^2(P\Sigma_n)$ is at most
\begin{eqnarray}
\#( \{ [ \xi_{ij}, \xi_{ji} ]| 1 \leq i < j \leq n \} \; &\cup& \{[ \xi_{ij}, \xi_{i,t} ]|  1\leq i\leq n, i\neq j,t, 1\leq j<t \leq n\}) \nonumber\\
&=&\frac{1}{2}n(n-1)^2= (n-1) \cdot r_n(2).
\end{eqnarray}
Thus, keeping in mind \eqref{Eq:A.19}, $\phi_2(P\Sigma_n)= (n-1) \cdot r_n(2).$
\end{proof}
As we have just seen, all the ranks of $gr^k(P\Sigma_n)$ determined here are obtained by the formula
\begin{eqnarray}
\phi_k(P\Sigma_{n})=(n-1)\cdot r_n(k)
\label{Eq:3.15}
\end{eqnarray}
We conclude this section by conjecturing that the formula \eqref{Eq:3.15} holds for all $n\geq 4 $ and $k\geq 3$ whose demonstration is the subject of a work in progress. 
\begin{conjecture}
	Let $k,n$ be integers such that $k \geq 3$ and $n \geq 4.$ The rank $\phi_k(P\Sigma_n)$ of $gr^k(P\Sigma_n)$ is given by $(n-1) \cdot r_n(k).$
\label{conjecturelower}	
\end{conjecture}

\subsection{Factor groups of the lower central series for upper triangular McCool group}
\label{sec:Factor groups of the lower central series for upper triangular McCool group}
Let $P\Sigma^+_n$ be the subgroup of $P{\Sigma}_n$ generated by $\xi_{i,j}$ with $1\leq j< i\leq n$. In \cite{cpvw}, it is shown that $P\Sigma^+_n$ can be realized as an iterated almost-direct product of free groups. If we now set $x_{p,i}:=\xi_{n-i+1,n-p}$ for all $1\leq i \leq p \leq n-1,$ then
\begin{equation}
P\Sigma^+_n=\mathbb{F}_{n-1} \rtimes_{\mu_{n-1}}(\mathbb{F}_{n-2} \rtimes_{\mu_{n-2}}( \dots ( \mathbb{F}_2 \rtimes_{\mu_{2}} \mathbb{F}_1)\dots )):=\rtimes^{n-1}_{p=1} \mathbb{F}_p
\label{decomposiupper}
\end{equation}
where $\mathbb{F}_p$ ($p=1,\dots,n-1$) is a free subgroup of $P\Sigma_n$ of rank $p$ on $\{x_{p,1},\dots,x_{p,p}\}$ and the homomorphism $\mu_{p}:\rtimes^{p-1}_{j=1}\mathbb{F}_j\rightarrow IA(\mathbb{F}_p)$ which determine the structure of this iterated semidirect product is given by :
\begin{equation}
\mu_{p} (x_{q,j})(x_{p,i})=x^{-1}_{q,j}\; x_{p,i}\; x_{q,j}=\left\{ \begin{array}{cl} x_{p,j}\; x_{p,i} \; x^{-1}_{p,j}   & \textrm{if } q =i \\ x_{p,i} & \text{ otherwise }   \end{array}\right.
\end{equation} where $1\leq i \leq p$, $1\leq j \leq q $ and $1 \leq q <p \leq n-1.$ As a consequence Prop.\ref{2.1.18}, we have the following finite presentation of $P\Sigma^+_{n}$ \cite{cohen}:
\begin{prop}
The upper triangular McCool group $P\Sigma^+_{n}=\rtimes^{n-1}_{p=1} \mathbb{F}_p$ where $\mathbb{F}_p=\langle x_{p,i}| 1\leq i \leq p \rangle$ has a finite presentation with generators $x_{p,i} (1\leq i \leq p \leq n-1)$ and with the following relations
	\begin{eqnarray}
	\left[x_{p,i}, x_{q,j}\right]&=&\left[x_{p,i}, x^{-1}_{p,j}  \right]\textrm{ if } q =i \nonumber\\
	\left[x_{p,i}, x_{q,j}\right]&=&\displaystyle{1\!\!1} \text{ otherwise } \label{Eq:3.18}
	\end{eqnarray}
	where $1\leq i \leq p$, $1\leq j \leq q $ and $1 \leq q <p \leq n.$
	\label{3.5}
\end{prop}
From the almost-direct product decomposition $P\Sigma_n^+=\rtimes^{n-1}_{p=1} \mathbb{F}_p$, we can describe the splitting $P\Sigma_n^+=\mathbb{F}_{n-1} \rtimes P\Sigma_{n-1}^+.$ Note that $P\Sigma_{n-1}^+$ acts by conjugation on $\mathbb{F}^{ab}_{n-1}$ trivially. It results from [\cite{ihara}, Lemma 3.11] that the almost-direct product $P\Sigma_n^+=\mathbb{F}_{n-1} \rtimes P\Sigma_{n-1}^+$ induces a decomposition of the associated Lie algebras over $\mathbb{Z}$
\begin{eqnarray}
gr^{\ast}(P\Sigma^+_{n})=gr^{\ast}(\mathbb{F}_{n-1}) \bigoplus gr^{\ast}(P\Sigma^+_{n-1}).
\end{eqnarray}
Let $\mathfrak{e}_{p,i}$ $(1\leq i \leq p \leq n-1)$ denote the image of each $x_{p,i}$ in $gr^{1}(P\Sigma^+_{n}).$ For each $1 \leq p \leq n-1,$ let $\mathfrak{E}_p=\{\mathfrak{e}_{p,1}, \dots, \mathfrak{e}_{p,p}\}$ and let $L[\mathfrak{E}_p]$ be the corresponding free Lie algebra of rank $p$ over $\mathbb{Z}.$ We know that $gr^{\ast}(\mathbb{F}_{n-1})$ is the free Lie algebra of rank $n-1$ over $\mathbb{Z}$ generated by the classes $\mathfrak{e}_{(n-1),1},\dots, \mathfrak{e}_{(n-1),(n-1)} \in gr^{1}(\mathbb{F}_{n-1})$ of $x_{(n-1),1},\dots, x_{(n-1),(n-1)}.$ Hence 
\begin{eqnarray}
gr^{\ast}(P\Sigma^+_{n})=L[\mathfrak{E}_{n-1}] \bigoplus gr^{\ast}(P\Sigma^+_{n-1}).
\label{58}
\end{eqnarray}
\begin{prop}
The Lie algebra $gr^{\ast}(P\Sigma^+_{n})$ is the quotient of the free Lie algebra $L[\mathfrak{E}_p]$ over $\mathbb{Z}$ generated by elements $\mathfrak{e}_{p,i}$ $(1\leq i \leq p \leq n-1)$ modulo the following relations
	\begin{eqnarray}
	\left[\mathfrak{e}_{p,i}, \mathfrak{e}_{p,j}+ \mathfrak{e}_{q,j}\right]&=&0 \textrm{ if } q =i \nonumber\\
	\left[\mathfrak{e}_{p,i}, \mathfrak{e}_{q,j}\right]&=&0 \text{ otherwise } \nonumber
	\end{eqnarray}
where $1\leq i \leq p$, $1\leq j \leq q $ and $1 \leq q <p \leq n.$	
\label{Lie:3.6}
\end{prop}
\begin{proof}
The proof adapts verbatim to the proof of Prop.\ref{Lie:3.9} by considering $P\Sigma^+_n$ instead of $I_n.$
\end{proof}
\section{Group cohomology of partial inner automorphism group $I_n$}
\label{chap: Anneau de cohomologie du groupe}

Here, we determine the structure of the cohomology ring $H^{\ast}(I_n,\mathbb{Z})$ of $I_n.$ Recall first that an almost-direct product $G=\rtimes^k_{p=1} G_p$ of the free groups $G_p= \langle x_{p,1}, \dots, x_{p,n_p} \rangle $ is an iterated semidirect product of free groups in which the action of the constituent free groups on the abelianization of one another is trivial. In \cite{cohen}, Cohen determined the structure of the cohomology ring $H^{\ast}(G,\mathbb{Z})$ of such a group as being a quotient of the outer algebra $E=\bigwedge H^{1}(G,\mathbb{Z}).$ Let $\{ e_{p,i} | \; 1\leq i\leq n_p , 1\leq p \leq k \}$ denote the dual basis of the basis $\{\overline{x_{p,i}} |1\leq i\leq n_p , 1\leq p \leq k \}$ of $H_1(G,\mathbb{Z})=G^{ab}.$ We state its result as follows: Let $E=\bigwedge H^{1}(G,\mathbb{Z})$ be the outer algebra generated by the elements $e_{p,i}$ $(1\leq i \leq n_p, \; 1\leq p \leq k)$ and let $\eta^{i,j}_p$ be the elements of the form $\eta^{i,j}_p=  e_{p,i}\wedge e_{p,j} +\sum_{q=1}^{p-1} \sum_{r=1}^{n_q} \sum_{s=1}^{n_p} \kappa^{i,j,r,s}_{q,p} e_{q,r}\wedge e_{p,s}$ where the coefficients $\kappa^{i,j,r,s}_{q,p}$ are the entries of a matrix (see for details \cite{cohen} .) The set $\mathfrak{J}=\{ \eta^{i,j}_p | 1\leq p \leq k, \; 1\leq i \leq j \leq n_p \} $ is a basis of $ker(H^2(\mathbb{Z}^m) \rightarrow H^2(G)),$ the kernel of the dual of homomorphism $H_2(\mathbb{Z}^m) \rightarrow H_2(G)$ with $m=\displaystyle\sum^{k}_{p=1}n_p.$ Let $J$ be the bilateral ideal generated by the elements of $\mathfrak{J}.$
The result of Cohen [Theorem 3.1 in \cite{cohen}] shows that $H^{\ast}(G,\mathbb{Z})$ is isomorphic to the quotient $E$ by $J.$ 
\begin{example}
	For the group $G=P\Sigma^+_n$ is generated by elements $x_{p,i}=\xi_{n-i+1,n-p}$ for $1\leq i \leq p \leq n-1$ subject the following relations 
	\begin{eqnarray}
	\left[x_{p,i}, x_{q,j}\right]&=&\left[x_{p,i}, x^{-1}_{p,j}  \right] \textrm{ if } q =i \nonumber\\
	\left[x_{p,i}, x_{q,j}\right]&=& \displaystyle{1\!\!1} \text{ otherwise } \nonumber
	\end{eqnarray}
	where $1\leq i \leq p$, $1\leq j \leq q $ and $1 \leq q <p \leq n,$ (see Prop.\ref{3.5} in Appendix A2). As a consequence [Theorem 3.1 in \cite{cohen}], we have $H^{\ast}(P\Sigma^{+}_{n},\mathbb{Z}) = E_P/J_P$ where $E_P= \bigwedge H^{1}(P\Sigma^{+}_{n},\mathbb{Z})$ is generated by $e_{p,i}, \; 1 \leq i < p \leq n-1$ with $e_{p,i}$ represents the dual of generators $\overline{x_{p,i}}$ of $H_1(P\Sigma^{+}_{n}, \mathbb{Z})$ and $J_P$ is the ideal generated by $$e_{p,i}\wedge e_{p,j}-e_{p,i}\wedge e_{i,j}, \text{ $ $ where }   j<i < p .$$ 
\end{example}$ $\\
This description of $H^{\ast}(P\Sigma^{+}_{n},\mathbb{Z})$ is exactly that given in \cite{cpvw}, but the proof is different. A similar result holds for the group $I_n.$ Let $\{a_{p,i} |\; 2\leq p \leq n, 1\leq i \leq p \}$ denote the dual basis of the basis $\{\overline{\nu}_{p,i} |\;2\leq p \leq n, 1\leq i \leq p \}$ of $H_1(I_n,\mathbb{Z}).$
\begin{prop}
The cohomology ring $H^{\ast}(I_n,\mathbb{Z})$ of $I_n$ is isomorphic to $E_I/J_I,$ where $E_I=\bigwedge H^1(I_n,\mathbb{Z})$ is the exterior algebra generated by degree-one element $a_{p,i}$ for $2\leq p \leq n, 1\leq i \leq p $ and $J_I$ is the ideal generated by the elements $$a_{p,i}\wedge a_{p,j}+a_{q,j}\wedge a_{p,i} \text{$ $  where }  j < i \leq q<p.$$
	\label{3.3.3}
\end{prop}
\begin{proof}
	Since $I_n$  is an almost-direct product of free groups, the cohomology $H^{\ast}(I_n,\mathbb{Z})$ of  $I_n$ follows directly from the results of [Theorem 3.1 in \cite{cohen}].
\end{proof}

\bibliographystyle{abbrv}	


\end{document}